\documentclass[12pt,amsymb,fullpage]{amsart}
\usepackage{amssymb,amscd,pstricks,amsmath}

\newtheorem{theorem}{Theorem}[section]
\newtheorem{defn}[theorem]{Definition}

\newtheorem{lemma}[theorem]{Lemma}

\newtheorem{eple}[theorem]{Example}
\newtheorem{rmk}[theorem]{Remarks}
\newtheorem{dsc}[theorem]{Discussion}
\newtheorem{nota}[theorem]{Notation}

\newsavebox{\indbin}
\savebox{\indbin}{\begin{picture}(0,0)
\newlength{\gnu}
\settowidth{\gnu}{$\smile$} \setlength{\unitlength}{.5\gnu}
\put(-1,-.65){$\smile$} \put(-.25,.1){$|$}
\end{picture}}

\newcommand{\be}{\begin{enumerate}}
\newcommand{\bd}{\begin{defn}}
\newcommand{\bt}{\begin{theorem}}
\newcommand{\bl}{\begin{lemma}}
\newcommand{\ee}{\end{enumerate}}
\newcommand{\ed}{\end{defn}}
\newcommand{\et}{\end{theorem}}
\newcommand{\el}{\end{lemma}}

\begin{document}
\title{Bounding the number of maximal torsion cosets on subvarieties of algebraic tori}
\author{Tristram de Piro}
\address{Mathematics Department, The University of Edinburgh, Kings Buildings, Mayfield Road, Edinburgh, EH9 3JZ} \email{depiro@maths.ed.ac.uk}
\thanks{The author was supported by the William Gordon Seggie Brown research fellowship}
\begin{abstract}
We obtain bounds on the number of maximal torsion cosets for
algebraic subvarieties $V\subset\mathbb{G}_{m}^{n}$, defined over ${\mathbb
Q}$, using model theoretic methods.
\end{abstract}
\maketitle

Let $V$ be an algebraic variety defined over the rationals ${\mathbb Q}$, with
$V\subset\mathbb{G}_{m}^{n}({\mathbb C})$, where $\mathbb{G}_{m}^{n}({\mathbb C})$ is the multiplicative subgroup of ${\mathbb C}^{n}$. A \emph{cyclotomic point} $\overline{\omega}$ on $V$ is a point of the form
$(\omega_{1},\ldots,\omega_{n})$, with $\omega_{i}$ a root of unity
for $1\leq i\leq n$. A \emph{torsion coset} is a set of the form
${\bar\omega}T$ where $T$ is a connected algebraic subgroup of
${\mathbb G}_{m}^{n}({\mathbb C})$. Any connected algebraic subgroup $T$ of
${\mathbb G}_{m}^{n}({\mathbb C})$ is isomorphic to ${\mathbb G}_{m}^{r}({\mathbb
C})$ for some $0\leq r\leq n$. We can write the isomorphism $\Phi$
in the following form;\\
\begin{equation*}
$$ $\Phi:{\mathbb G}_{m}^{r}({\mathbb C})\rightarrow T$\\

$:(t_{1},\ldots,t_{r})\mapsto (t_{1}^{m_{11}}\ldots
t_{r}^{m_{1r}},\ldots,t_{1}^{m_{n1}}\ldots t_{r}^{m_{nr}})$ $$
\end{equation*}

We let $\textbf{M}=(m_{ij})_{1\leq i\leq n}^{1\leq j\leq r}$ be
the matrix defining the isomorphism $\Phi$, so we can write an
element of a torsion coset in the form $\bar\omega
{\textbf{t}}^{\textbf{M}}$, where
$\textbf{t}=(t_{1},\ldots,t_{r})\in {\mathbb G}_{m}^{r}(\mathbb C)$. We
define a torsion coset $\bar\omega T\subset V$ to be \emph{maximal} if
for any torsion coset $\bar\omega 'T'\subset V$, with $\bar\omega
T\subset\bar\omega 'T'$, we have that $\bar\omega T=\bar\omega'
T'$.\\

 We make the following straightforward observations about
maximal torsion cosets. First,\\

\begin{lemma}

${\overline {V(K^{cycl})}}=\bigcup_{i=1}^{N}\bar\omega_{i}T_{i}$\\

where $i$ runs over all the maximal torsion cosets on $V$.

\end{lemma}

 This is
clear by noting that a cyclotomic point on $V$ is itself a torsion
coset, hence we obtain the left hand inclusion. We obtain equality
by the fact that cyclotomic points are Zariski dense in any
torsion coset. Secondly, if ${\bar a} S\subseteq V$ is a coset of
a connected algebraic subgroup of ${\mathbb G}_{m}^{n}$, then;\\

$\overline{{\bar a} S(K^{cycl})}=\bar w T$\\

where $T$ is a connected algebraic subgroup of ${\mathbb G}_{m}^{n}(K)$ and
$\bar w$ is a cyclotomic point.\\

In order to see this, observe that $\overline{{\bar
a}S(K^{cycl})}=\bigcup_{i=1}^{N}\bar\omega_{i}T_{i}$, where $i$
runs over all the maximal torsion cosets on $\bar a S$. We claim
that in fact $N=1$. Let $\bar w$ be a cyclotomic point on ${\bar
a} S$, then we have that ${\bar w_{1}}^{-1}\bar w\in S$ and $\bar
w\in\bar w_{1}S$. Hence ${\bar a}S(K^{cycl})=({\bar
w_{1}}S)(K^{cycl})={\bar w_{1}}(S(K^{cycl}))$. We now claim that
$\overline{S(K^{cycl})}$ is connected, this follows easily from
the fact that $S$ is a connected algebraic subgroup of
${\mathbb G}_{m}^{n}(K)$ and the classification of such groups. It follows
that $\overline{{\bar a}S(K^{cycl})}={\bar w_{1}}T_{1}'$, but, by
maximality of ${\bar w_{1}}T_{1}$, we have that $T_{1}=T_{1}'$ as
required.\\

An important consequence of the above result is the following;\\

\begin{lemma}

Let $\bar w_{i}T_{i}$, for $1\leq i\leq n$, enumerate the maximal
torsion cosets on $V$ and suppose that;\\

$\bigcup_{i=1}^{N}\bar w_{i}T_{i}\subseteq\bigcup_{i=1}^{M}\bar
a_{i}S_{i}\subseteq V$\\

where the $S_{i}$ are connected algebraic subgroups of
${\mathbb G}_{m}^{n}(K)$. Then $N\leq M$.

\end{lemma}

\begin{proof}

We have that $\bigcup_{i=1}^{N}\bar
w_{i}T_{i}=\overline{\bigcup_{i=1}^{M}\bar
a_{i}S_{i}(K^{cycl})}=\bigcup_{i=1}^{M}\overline{a_{i}S_{i}(K^{cycl})}$.
By the previous argument, this is a union of $M$ irreducible
torsion cosets. As the irreducible decomposition of
$\bigcup_{i=1}^{N}\bar w_{i}T_{i}$ is irredundant, by maximality
of each torsion coset $\bar w_{i}T_{i}$, we have that $N\leq M$ as required.\\

\end{proof}

The following result is given in \cite{Sc}, (Theorem
$8E^{*}$);\\

Suppose that $V$ is defined by equations of total degree $\leq d$,
then there exist finitely many maximal torsion cosets on $V$ and,
moreover, the number is bounded by $exp(3N(d)^{3\over
2}log(N(d)))$ where $N(d)=C^{n+d}_{d}$. \\

The purpose of this paper is to find an alternative
bound in the degree $d$ and dimension $n$, using methods from model theory.\\

We define the exponent of a torsion coset $\bar\omega T$ to be any
multiple of its order as an element of the group ${{\mathbb G}_{m}^{n}/T}$.
By the result given above, there must exist a single exponent for
all maximal torsion cosets on $V$. In \cite{Mc}, (Theorem
4.2), the following result is proved;\\

Suppose that $N(V)$, the Newton polygon associated to $V$, has
diameter $D(V)$. Then every $(n-k)$-dimensional maximal torsion
coset on $V$ has an exponent $mP_{N}$ for $m\leq
D(V)^{2k}k^{k/2}$, where $N=Card(N(V))\leq C^{n+d}_{d}$ and
$P_{N}$ is the product of all primes up to $N$. (By results of
Iskander Aliev, the bound on $m$ can be improved to $D(V)^{2k}$.)\\

Now let $K$ be a uniform exponent for all the maximal torsion
cosets on $V$, for example we can take $K=tP_{N}$, where $t\leq
D(V)^{n(n-1)}$, assuming the result of Aliev. Suppose that we are
given a maximal torsion coset $\bar \omega T$ of exponent $K$,
then we have that ${\bar \omega}^{K}\in T$. Suppose that
$\Phi(\bar t)={\bar \omega}^{K}$. Clearly, ${\mathbb G}_{m}^{r}(\mathbb C)$
is closed under taking $K$'th roots, so we can find $\bar t_{1}\in
{\mathbb G}_{m}^{r}(\mathbb C)$ such that ${\bar t_{1}}^{K}=\bar t$. Then
$\Phi({\bar t_{1}}^{K})=\Phi(\bar t_{1})^{K}={\bar \omega}^{K}$.
Therefore, $(\Phi(\bar t_{1})^{-1}\bar \omega)^{K}=\textbf{1}$ and
$\Phi(\bar t_{1})^{-1}\bar \omega$ represents the same torsion
coset as $\bar \omega$. Enumerate elements $\{{\bar
w}_{1},\ldots,{\bar w}_{N}\}$ representing the maximal torsion
cosets on $V$ such that ${\bar w}_{j}^{K}=\textbf{1}$ for $1\leq
j\leq N$. We can write each ${\bar w}_{j}$ as
$(\omega_{1j},\ldots,\omega_{nj})$ where $\omega_{ij}$ is a
primitive $L_{ij}$'th root of unity with $L_{ij}|K$, for $1\leq
i\leq n$. Now choose a primitive $K$'th root of unity $\xi$ with
${\xi}^{K/L_{ij}}=\omega_{ij}$. We consider the following
cyclotomic extensions;\\

(i). ${\mathbb Q}(\xi)/{\mathbb Q}$;\\

In this case $Gal({\mathbb Q}(\xi)/{\mathbb
Q})=U_{p_{1}^{m_{1}}}\times\ldots\times U_{p_{r}^{m_{r}}}$, where
$K=p_{1}^{m_{1}}\ldots p_{r}^{m_{r}}$ is the prime factorisation
of $K$ and $U_{p_{j}^{m_{j}}}$ is the cyclic group of units in the
multiplicative group $({\mathbb Z}/{p_{j}^{m_{j}}\mathbb
Z})^{*}$. This is an abelian group with generators
$\{\sigma_{1},\ldots,\sigma_{r}\}$.\\

a. If $K$ is odd, $(2,K)=1$, hence $2\in U_{K}$ and the map
$\sigma(\xi)=\xi^{2}$ determines an element of $Gal({\mathbb
Q}(\xi)/{\mathbb Q})$. We then have that
$\sigma(\omega_{ij})=\omega_{ij}^{2}$ as well.\\

b.  If $K=2L$ with $L$ odd, then $(L+2,2L)=1$ and
$\sigma(\xi)={\xi}^{L+2}=-\xi^{2}$ determines an element of
$Gal({\mathbb Q}(\xi)/{\mathbb Q})$, then
$\sigma(\omega_{ij})=(-1)^{K/L_{ij}}\omega_{ij}^{2}$.\\

c.  If $K=4L$, then $(2L+1,4L)=1$ and
$\sigma(\xi)={\xi}^{2L+1}=-\xi$ determines an element of
$Gal({\mathbb Q}(\xi)/{\mathbb Q})$, then
$\sigma(\omega_{ij})=(-1)^{K/L_{ij}}\omega_{ij}$.\\

We extend $\sigma$ to a generic automorphism of ${\mathbb Q}(\xi)\subset K$,
hence $(K,\sigma)\models ACFA$.

In the three cases, we construct a functional equation in
$\sigma$. We denote the coordinates on ${\mathbb G}_{m}^{n}$ by
$(x_{1},\ldots,x_{n})$\\

(a,b).$((\sigma(x_{1})-x_{1}^{2})(\sigma(x_{1})+x_{1}^{2}),\ldots,(\sigma(x_{n})-x_{n}^{2})(\sigma(x_{n})+x_{n}^{2}))=\textbf{0}$\ $\textbf{(1)}$\\

(c).
$((\sigma(x_{1})-x_{1})(\sigma(x_{1})+x_{1}),\ldots,(\sigma(x_{n})-x_{n})(\sigma(x_{n})+x_{n}))=\textbf{0}$\\
(*)

If we denote the action by $\sigma$ on ${\mathbb G}_{m}^{n}(K)$ as
$\sigma(x_{1},\ldots x_{n})=(\sigma(x_{1}),\ldots,\sigma(x_{n}))$,
and for ${m\in\mathcal Z}$, define
$m(x_{1},\ldots,x_{n})=(x_{1}^{m},\ldots,x_{n}^{m})$, then any
polynomial $P(X)$ with integer coefficients defines an
endomorphism of ${\mathbb G}_{m}^{n}(K)$. We can then write the functional
equation\\

$(\sigma(x_{1})-x_{1}^{p},\ldots,\sigma(x_{n})-x_{n}^{p})=\textbf{0}$\\

as $Ker(P(\sigma))$ where $P(X)$ denotes the polynomial $X-p$.\\

We will denote the subgroups of ${\mathbb G}_{m}^{n}(K)$ defined by the
polynomial $X-p$ for $p\in {\mathbb Z}$ by ${\mathbb G}_{p}$.\\

The functional equations from $(*)$ define subgroups of
${\mathbb G}_{m}^{n}(K)$, which we denote by ${\mathbb G}^{1}$ and ${\mathbb G}^{2}$.\\

We need the following definitions;\\

Let $A$ be a $\sigma$-definable subgroup of ${\mathbb G}_{m}^{n}(K)$. We say
that $A$ is LMS (stable, stably embedded and 1-based) if every
$\sigma$-definable subset $X$ of $A^{r}$ (possibly with parameters
outside A) is a finite Boolean combination of cosets of definable
subgroups of $A^{r}$. We say that $A$ is algebraically modular
(ALM), if for every $\sigma$ definable $X\subset A^{r}$, the
Zariski closure of $X$ is a finite union of cosets of algebraic subgroups of ${\mathbb G}_{m}^{n}(K)$\\

The following result is proved in \cite{Hr} (Corollary 4.1.13);\\

Let $p(T)$ be a polynomial with integer coefficients defining an
endomorphism $p(\sigma)$ of ${\mathbb G}_{m}^{n}(K)$. Then $Ker(p(\sigma))$
is LMS iff $p(T)$ has no cyclotomic factors. \\

Applying this result to the polynomial $X-p$ for $p\in {\mathbb
Z}$ gives that the groups ${\mathbb G}_{p}$ are LMS iff $p\notin\{1,-1\}$.\\

If $A$ is LMS, then $A$ is ALM. This is almost immediate from the
definitions. Let $X\subset A^{r}$ be $\sigma$-definable. As $A$ is
LMS, we can write;\\

$X=\bigcup_{i=1}^{n}(C_{i}\setminus D_{i})$\\

where the $C_{i}$ are disjoint cosets of groups $H_{i}\subset
A^{r}$ and $D_{i}$ is contained in a finite union of cosets of
subgroups of $H_{i}$ of infinite index. We then clearly have that;\\

$\overline{X}=\overline{\bigcup_{i=1}^{n}(C_{i}\setminus
D_{i})}=\bigcup_{i=1}^{n}\overline{C_{i}}$\\

Each $\overline C_{i}$ is clearly an algebraic subgroup of
${\mathbb G}_{m}^{n}(K)$, which gives the result.\\

In particular, the groups ${\mathbb G}_{p}$ are ALM for $p\notin \{1,-1\}$.\\

We claim that the group ${\mathbb G}^{1}$ is also ALM. We have that
${\mathbb G}^{1}=\bigcup_{\delta_{1},\ldots,\delta_{n}}W_{\delta_{1},\ldots,\delta_{n}}$
for $\delta_{j}\in\{0,1\}$where $W_{\delta_{1},\ldots,\delta_{n}}$
is the $\sigma$-variety defined by the functional equation\\

$(\sigma(x_{1})+(-1)^{\delta_{1}}x_{1}^{2},\ldots,\sigma(x_{n})+(-1)^{\delta_{n}}x_{n}^{2})=\textbf{0}$\\

Assuming $W_{\delta_{1},\ldots,\delta_{n}}\neq\emptyset$, we can
find $\bar w\in W_{\delta_{1},\ldots,\delta_{n}}\neq\emptyset$,
then the map $\theta:W_{\delta_{1},\ldots,\delta_{n}}\rightarrow
{\mathbb G}_{2}$ given by
$\theta(x_{1},\ldots,x_{n})=(w_{1}x_{1},\ldots,w_{n}x_{n})$ is
easily checked to be a definable bijection. It follows easily that
the property of ALM is inherited by each
$W_{\delta_{1},\ldots,\delta_{n}}$, hence by ${\mathbb G}^{1}$.\\

We now claim the following for the situations a. and b.;\\

$\bigcup_{1\leq j\leq N}\bar w_{j}T_{j}\subseteq \overline{V\cap
{\mathbb G}^{1}}=\bigcup_{1\leq j\leq M}\bar a_{j}T_{j}$, where the left
hand side consists of the union of all maximal \emph{torsion}
cosets on $V$ and the right hand side consists of a union of
cosets of algebraic subgroups of ${\mathbb G}_{m}^{n}(K)$. By the property
of ALM, we clearly have the right hand equality. For the left hand
inclusion, suppose that $\bar w_{j}T_{j}$ defines a maximal
torsion coset on $V$, where $\bar w_{j}$ is chosen in the form
given above. Then, by construction, $\bar w_{j}\in {\mathbb G}^{1}$. Now
consider the variety $W\subset {\mathbb G}_{m}^{r}(K)\times {\mathbb G}_{m}^{r}(K)$
defined by the equations
$<y_{1}-x_{1}^{2},\ldots,y_{r}-x_{r}^{2}>$. This has dimension $r$
and projects dominantly onto the factors ${\mathbb G}_{m}^{r}(K)$. We claim
that there exists $(x_{1},\ldots,x_{r})\in {\mathbb G}_{m}^{r}(K)$, generic
over $acl(Q)$, with $(\bar x,\sigma(\bar x))\in W$. By
compactness, it is sufficient to prove that for any proper closed
subvariety $Y$ of ${\mathbb G}_{m}^{r}(K)$, defined over $acl(\mathcal Q)$,
we can find $\bar x$ with $(\bar x,\sigma(\bar x))\in W\setminus
(Y\times {\mathbb G}_{m}^{r})\cap W$. This can be done exactly using the
axiom scheme for $ACFA$. Now let $\Phi$ define the isomorphism of
$T_{j}$ with ${\mathbb G}_{m}^{r}(K)$, then $\bar w_{j}\Phi(\bar x)$ belongs
to ${\bar w_{j}}T_{j}$ and is generic over the field defining the
union of the maximal torsion cosets. By construction, the point
$\bar w_{j}\Phi(\bar x)$ lies inside $V\cap {\mathbb G}^{1}$. Hence, $\bar
w_{j}T_{j}\subseteq\overline{V\cap {\mathbb G}^{1}}$ as required.\\

The functional equation $(c)$ does not define an ALM
$\sigma$-variety, hence this approach fails (can it be defined
using exponents $\geq 1$??)\\

Instead, we can use the method given in \cite{Hr} to handle this
case;\\

Fix prime numbers $p$ and $q$ with $p\neq q$. We claim there
exists $\sigma\in Gal({\overline {\mathbb Q}}/{\mathbb Q})$ with
$\sigma(\omega)=\omega^{p}$ for all primitive $K$'th roots of
unity with $(K,p)=1$ and $\sigma(\omega)=\omega^{q}$ for all
primitive $K$'th roots with $K=p^{s}$. Let ${\mathbb Q}_{p}$ be
the completion of ${\mathbb Q}$ at the prime $p$ and
$Q_{p}^{unr}=\bigcup_{(K,p)=1}{\mathbb Q}(\xi_{K})$ with
$\xi_{K}$ a primitive $K$'th root of unity. The residue field of
${\mathbb Q}_{p}(\xi_{K})$ is $F_{p}(\xi_{K})$. The extension is
unramified, so $Gal({\mathbb Q_{p}}(\xi_{K})/{
Q})=Gal(F_{p}(\xi_{k})/F_{p})$, with canonical generator
$\phi_{p}$ lifting Frobenius given by
$\phi_{p}(\xi_{K})=\xi_{K}^{p}$. Clearly $\phi_{p}$ lifts to an
element of $Gal({\mathbb Q}_{p}^{unr}/{\mathcal Q_{p}})$ such
that $\phi_{p}(\xi_{K})=\xi_{K}^{p}$ for all primitive $K$'th
roots with $(K,p)=1$. Similarily, we can find $\phi_{q}$ with
$\phi_{q}(\xi_{K})=\xi_{K}^{q}$ for all primitive $K$'th roots
with $(K,q)=1$. As $p\neq q$, we have
$\phi_{q}(\xi_{K})=\xi_{K}^{q}$ for $K=p^{s}$. By restriction,
$\phi_{p}$ and $\phi_{q}$ define automorphisms of $L$ and $M$
where $L$ is the extension of ${\mathbb Q}$ obtained by adding
primitive roots of unity $\xi_{K}$ with $(K,p)=1$ and $M$ is
obtained by adding primitive roots of unity with $K=p^{s}$. These
fields are linearly disjoint over ${\mathbb Q}$, hence we can
find a single automorphism $\sigma$ on $\overline {\mathbb Q}$
with $\sigma|K=\phi_{p}$ and $\sigma|L=\phi_{q}$. As usual, we can
extend $\sigma$ to a generic automorphism of $K$. Now we consider
the endomorphism on ${\mathbb G}_{m}^{n}(K)$ defined by $p(\sigma)$ where
$p(X)$ is the polynomial $(X-p)(X-q)$. We claim that
${\mathbb G}_{m}^{n}(K^{cycl})\subset Ker(p(\sigma))$. First note that if
$\bar w$ is a cyclotomic point of order $K$ with $(K,p)=1$, then
we can write $\bar w=(\omega_{1},\ldots,\omega_{n})$ with the
$\omega_{i}$ primitive roots of unity of order prime to $p$. By
construction $\sigma(\bar w)=p(\bar w)$, hence $(\sigma-p).(\bar
w)=\textbf{1}$. By the same argument, if $\bar w$ is a cyclotomic
point of order $p^{s}$, then $(\sigma-q).(\bar w)=\textbf{1}$. Now
suppose that $\bar w$ is an arbitrary cyclotomic point of order
$L$. Let $L=p^{s}K$ where $(K,p)=1$. We can find integers $a$ and
$b$ with $aK+bp^{s}=1$, hence $\bar w=(\bar w)^{aK}(\bar
w)^{bp^{s}}=(\bar w_{1})^{a}(\bar w_{2})^{b}$ where $\bar w_{1}$
is cyclotomic of order $p^{s}$ and $\bar w_{2}$ is cyclotomic of
order $K$. We then have that;\\

$p(\sigma)(\bar w)=(p(\sigma)(\bar w_{1}))^{a}(p(\sigma)(\bar
w_{2}))^{b}=((\sigma-p).\textbf{1})^{a}.((\sigma-q).\textbf{1})^{b}=\textbf{1}$\\

as required. We now have the following explicit functional
equation for $Ker(p(\sigma))$ using coordinates
$(x_{1},\ldots,x_{n})$ on ${\mathbb G}_{m}^{n}(K)$;\\

$\ \ \ \ \ \ \ p(\sigma).(x_{1},\ldots,x_{n})=\textbf{1}$\ \ \ \ \ \ \ \textbf{(2)}\\

$\Longleftrightarrow
(\sigma-p).{(\sigma(x_{1}),\ldots,\sigma(x_{n}))\over
(x_{1}^{q},\ldots,x_{n}^{q})}=\textbf{1}$\\

$\Longleftrightarrow
{(\sigma^{2}(x_{1}),\ldots,\sigma^{2}(x_{n})).(x_{1}^{pq},\ldots,x_{n}^{pq})\over
((\sigma(x_{1}))^{q},\ldots,((\sigma(x_{n}))^{q})).((\sigma(x_{1}))^{p},\ldots,((\sigma(x_{n}))^{p}))}=\textbf{1}$\\

$\Longleftrightarrow((\sigma^{2}x_{1})(x_{1}^{pq})-(\sigma
x_{1})^{p}(\sigma
x_{1})^{q},..,(\sigma^{2}x_{n})(x_{n}^{pq})-(\sigma
x_{n})^{p}(\sigma x_{n})^{q})=\textbf{0}$\\

Let ${\mathbb G}_{p,q}=Ker(p(\sigma))$. Using the theorem above, we know
that ${\mathbb G}_{p,q}$ is ALM. Moreover, we again have that;\\

$\bigcup_{1\leq j\leq N}\bar w_{j}
T_{j}\subseteq\overline{V\cap {\mathbb G}_{p,q}}=\bigcup\bar a_{j}T_{j}$. (**)\\

where the left hand union is over all maximal \emph{torsion}
cosets and the right hand union consists of cosets of algebraic
subgroups of ${\mathbb G}_{m}^{n}(K)$. We again have the right hand equality
by the ALM property. For the left hand inclusion, note that
$\bigcup_{1\leq j\leq N}\bar w_{j} T_{j}=\overline {V(K^{cycl})}$
and $V(K^{cycl})\subset V\cap
{\mathbb G}_{p,q}$ by construction of ${\mathbb G}_{p,q}$.\\

We now use the functional equations to obtain explicit bounds on
the number of maximal torsion cosets on $V$. This is done by
finding a bound $N$ for the number of irreducible components of
${\overline {V\cap G}}$ where $G$ is one of the groups
${\mathbb G}_{p,q},{\mathbb G}^{1}$.  By $(**)$ and Lemma 0.1, this will give a bound
$N$ for the number of maximal torsion cosets on $V$. We require
the following lemma (Proposition 2.2.1 in
\cite{Hr})\\

\begin{lemma}

Let $P_{n}(K)$ be $n$-dimensional projective space, and
$(K,\sigma)$ a difference closed difference field. Let $S$ be a
subvariety of $P_{n}^{l}(K)$ defined over $K$. Let \\

$Z$=Zariski closure of $\{x\in
P_{n}(K):(x,\sigma(x),\ldots,\sigma^{l-1}(x))\in S\}$.\\

Then $deg(Z)\leq deg(S)^{2^{dim(S)}}$. In particular $Z$ has at
most $deg(S)^{2^{dim(S)}}$ irreducible components.\\

\end{lemma}

Here, $deg(S)$ is the sum of the multi-degrees of $S$. It is a
straightforward exercise to rephrase this result replacing
${\mathbb P}_{n}(K)$ by ${\mathbb G}_{m}^{n}(K)$ and ${\mathbb P}_{n}^{l}(K)$ by
${\mathbb G}_{m}^{nl}(K)$.\\

Now, in the case of the functional equation \textbf{(1)} for the
cases $(a,b)$ and the functional equation \textbf{(2)} which
covers cases $(a,b,c)$, we construct the following varieties
$W_{\textbf{1}}$ and $W_{\textbf{2}}$.\\

For coordinates $(x_{1},\ldots,x_{n},y_{1},\ldots,y_{n})$ on
${\mathbb G}_{m}^{2n}(K)$;\\

$W_{\textbf{1}}=<(y_{1}-x_{1}^{2})(y_{1}+x_{1}^{2}),\ldots,(y_{n}-x_{n}^{2})(y_{n}+x_{n}^{2})>$\\

For coordinates
$(x_{1},\ldots,x_{n},y_{1},\ldots,y_{n},z_{1},\ldots,z_{n})$ on
${\mathbb G}_{m}^{3n}(K)$;\\

$W_{\textbf{2}}=<z_{1}x_{1}^{pq}-y_{1}^{p+q},\ldots,z_{n}x_{n}^{pq}-y_{n}^{p+q}>$\\

A straightforward calculation gives that
$deg(W_{\textbf{1}})=3^{n}$ and $dim(W_{\textbf{1}})=n$ whereas
$deg(W_{\textbf{2}})=(pq+1)^{n}$ and $dim(W_{\textbf{2}})=2n$\\

As $\sigma$ fixes $V$, we have in cases $(a,b)$ that;\\

$\bigcup_{1\leq j\leq N}\bar w_{j} T_{j}\subseteq\overline{\{x\in
{\mathbb G}_{m}^{n}(K):(x,\sigma x)\in V^{2}\cap W_{\textbf{1}}\}}$\\

and in cases $(a,b,c)$ that;\\

$\bigcup_{1\leq j\leq N}\bar w_{j} T_{j}\subseteq\overline{\{x\in
{\mathbb G}_{m}^{n}(K):(x,\sigma x,\sigma^{2}x)\in V^{3}\cap
W_{\textbf{2}}\}}$\\

We finally need the following version of Bezout's theorem for
counting the components of intersections in multi-projective
space;\\

\begin{lemma}{Bezout's Theorem}

Let $V,W$ be subvarieties of ${\mathbb P}_{n}^{l}(K)$. Let
$Z_{1},\ldots,Z_{t}$ be the irreducible components of $V\cap W$.
Then $\sum_{i=1}^{t}deg(Z_{i})\leq deg(V)deg(W)$.

\end{lemma}

Applying this to the current situation, we have $deg(V^{2}\cap
W_{\textbf{1}})\leq deg(V)^{2}3^{n}$ and $deg(V^{3}\cap
W_{\textbf{2}})\leq deg(V)^{3}(pq+1)^{n}$. We also have that
$dim(V^{2}\cap W_{\textbf{1}})\leq 2dim(V)$ and $dim(V^{3}\cap
W_{\textbf{2}})\leq 3dim(V)$. Now, using Lemma 0.2, we have
that;\\

$N\leq (deg(V)^{2}3^{n})^{2^{2dim(V)}}=exp[2^{2dim(V)}log(d^{2}3^{n})]$\ in cases $(a,b)$\\

and \\

$N\leq (deg(V)^{3}(pq+1)^{n})^{2^{3dim(V)}}$\ in cases
$(a,b,c)$.\\

where $d=deg(V)$\\

Taking $p=2$ and $q=3$ gives \\

$N\leq (deg(V)^{3}7^{n})^{2^{3dim(V)}}=exp[2^{3dim(V)}log(d^{3}7^{n})]$\ in cases $(a,b,c)$.\\

Using the toric version of Bezout's theorem in \cite{Fu}, we can
replace $deg(V)$ by $vol(N(V))$, where $N(V)$ is the Newton
polytope associated to $V$.This gives a slightly more refined
estimate for subvarieties $V$ of ${\mathbb G}_{m}^{n}(K)$. A comparison of
the estimates from \cite{Sc} shows that the estimate in this paper
is better in certain situations and worse in others,
depending on the dimension and degree of $V$.\\


\begin{thebibliography}{99}

\bibitem{Fu} W. Fulton, Introduction to Toric Varieties,
Princeton University Press, 1993.\\

\bibitem{Hr} E. Hrushovski, The Manin-Mumford Conjecture and the model
 theory of difference fields, Annals of Pure and Applied Logic 112
(2001) 43-115\\


\bibitem{Mc} J. McKee, C. Smyth, There are Salem Numbers of Every
Trace, Bulletin of the London Mathematical Society 37(2005)
25-26\\

\bibitem{Sc} W. Schmidt, Heights of Points on Subvarieties of
$G_{m}^{n}$ Number Theory, Paris (1993-1994), 157-187, London
Mathematical Society Lecture Notes Series, 235, Cambridge
University Press, Cambridge, 1996\\



\end{thebibliography}
\end{document}